\documentclass[12pt, twoside]{article}
\usepackage{amsmath,amsthm,amssymb}
\usepackage{times}
\usepackage{enumerate}
\usepackage{bm}
\usepackage[all]{xy}
\usepackage{mathrsfs}
\usepackage{amscd}
\usepackage{color}
\def\titlerunning#1{\gdef\titrun{#1}}
\makeatletter
\def\author#1{\gdef\autrun{\def\and{\unskip, }#1}\gdef\@author{#1}}
\def\address#1{{\def\and{\\\hspace*{18pt}}\renewcommand{\thefootnote}{}%
		\footnote {#1}}%
	\markboth{\autrun}{\titrun}}
\makeatother
\def\email#1{e-mail: #1}

\def\keywords#1{\par\medskip
	\noindent\textbf{Keywords.} #1}

\newtheorem{theorem}{Theorem}[section]
\newtheorem{corollary}[theorem]{Corollary}
\newtheorem{lemma}[theorem]{Lemma}
\newtheorem{proposition}[theorem]{Proposition}
\theoremstyle{definition}
\newtheorem{definition}[theorem]{Definition}
\newtheorem{remark}[theorem]{Remark}

\theoremstyle{example}
\newtheorem{example}[theorem]{Example}
\numberwithin{equation}{section}

\xyoption{all}

\def \a {\alpha }
\def \b {\beta}

\def \de {\delta}
\def \De {\Delta}
\def \la {\lambda}
\def \La {\Lambda}
\def\w {\omega}
\def\Om{\Omega}

\def\pa{\partial}
\def\na {\nabla}

\begin{document}
	\baselineskip=17pt
	
	\titlerunning{$L^{2}$ harmonic forms on complete special holonomy manifolds}
	\title{$L^{2}$ harmonic forms on complete special holonomy manifolds}
	
	\author{Teng Huang}
	
	\date{}
	
	\maketitle
	
	\address{T. Huang: School of Mathematical Sciences, Soochow University, Suzhou, 215006 and School of Mathematical Sciences, University of Science and Technology of China, Hefei, 230026, People's Republic of China; \email{htmath@ustc.edu.cn; htustc@gmail.com}}
	
	\begin{abstract}
		In this article, we consider $L^{2}$ harmonic forms on a complete non-compact Riemannian manifold $X$ with a nonzero parallel form $\w$. The main result is that if $(X,\w)$ is a complete $G_{2}$- ( or $Spin(7)$-) manifold with a $d$(linear) $G_{2}$- (or $Spin(7)$-) structure form $\w$, then the $L^{2}$ harmonic $2$-forms on $X$ vanish. As an application, we prove that the instanton equation with square integrable curvature on $(X,\w)$ only has trivial solution. We would also consider the Hodge theory on the principal $G$-bundle $E$ over $(X,\w)$.
	\end{abstract}
	\keywords{$L^{2}$ harmonic form; $G_{2}$- ($Spin(7)$-) manifold; $d$(linear)-form; gauge theory}
	\section{Introduction}
	Let $X$ be a $C^{\infty}$-manifold equipped with a differential form $\w$. This form is called parallel if $\w$ is preserved by the Levi-Civita connection: $\na\w=0$. This identity gives a powerful restriction on the holonomy group $\rm{Hol}(X)$. In K\"{a}hler geometry the parallel forms are the K\"{a}hler form and its powers. The algebraic geometers obtained many results of topological and geometric on studying the corresponding algebraic structure. In $G_{2}$- or $Spin(7)$-manifold, the parallel form is the $G_{2}$- or $Spin(7)$-structure. In \cite{Ver}, the author had  generalized some of these results on K\"{a}hler manifolds to other manifolds with a parallel form, especially the parallel $G_{2}$-manifolds. The results which obtained on \cite{Ver} can be summarized as K\"{a}hler identities for $G_{2}$-manifolds.
	
	The theory of $G_{2}$-manifolds is one of the places where mathematics and physics interact most strongly \cite{KL,LL}. In string theory, $G_{2}$-manifolds are expected to play the same role as Calabi-Yau manifolds in the usual A- and B-models of type-II string theories. There are many results on the construction of $G_{2}$-manifolds \cite{Bry,Joy1,Joy2,Kov}. Hitchin constructed a geometry flow \cite{Hit2} which physicists called Hichin's flow; it turned out to be extremely important in string physics.
	
	A basic question, pertaining both the function theory and topology on $X$, is: when are there nontrivial harmonic forms on $X$? When $X$ is not compact, a growth condition on the harmonic forms at infinity must be imposed, in order that the answer to this question be useful. A natural growth condition is square integrable; if $\La^{p}_{(2)}(X))$ denotes the $L^{2}$ $p$-forms on $X$ and $\mathcal{H}^{p}_{(2)}(X)$ the harmonic forms in $\La^{p}_{(2)}(X)$. One version of this basic question is: what is the structure of $\mathcal{H}^{p}_{(2)}(X)$? The study of $L^{2}$ harmonic forms on a complete Riemannian manifold is a very interesting and important subject; it also has numerous applications in the field of Mathematical Physics (see for example \cite{Hit}).
	
	In \cite{Gro}, Gromov states that if the K\"{a}hler form $\w$ on a complete K\"{a}hler satisfies $\w=\rm{d}\theta$, where $\theta$ is a bounded one-form, the only $L^{2}$ harmonic forms lie in the middle dimension. There are many complete K\"{a}hler manifolds with a exact K\"{a}hler form $\w$ \cite{CX,Gro,JZ}. In \cite{CX,JZ},\ they extended Gromov's theorem to the case of the one form $\theta$ is linear growth.
	
	A $G_{2}$- or a $Spin(7)$-structure of a $7$-, $8$-manifold is given by a parallel $3$-form $\phi$ or $4$-form $\Om$ (See \cite{Joy2} Section 10). It is also very intriguing to construct some examples of the $G_{2}$- or $Spin(7)$-manifolds with a $d$(linear) structure form. 
	\begin{example}
		There are some trivial examples of $G_{2}$-manifolds and $Spin(7)$-manifolds satisfy the growth conditions required.\\
		(1) Let $X$ be a complete connected manifold with zero sectional curvature and $\w$ be a parallel differential $k$-form on $X$, then the Theorem 1.1 \cite{CX} states that $\w$ is $d$(linear). But the Killing-Hopf theorem states that $X$ is isometric to a quotient of a Euclidean space by a group acting freely and properly discontinuously.\\
		(2) $C(X)=(\mathbb{R}^{+}\times X,\bar{\rm{g}})$ with $\rm{\bar{g}}=dr^{2}+r^{2}\rm{g}$ as the Riemannian or metric cone over $X$.
		It is well known that $X$ admits a real Killing spinor if and only if $C(X)$ admits a parallel spinor. Then, $C(X)$ has restricted holonomy and for any nearly K\"{a}hler 6-manifold $X$, $C(X)$ has holonomy $G_{2}$ and  for any nearly parallel $G_{2}$-manifold $X$, $C(X)$ has holonomy $Spin(7)$. We can show that the cone $C(X)$ is also the model for the growth conditions required (See Section 3). But following Hopf-Rinow theorem, it implies that the cone manifold $C(X)$ is noncomplete.
	\end{example}
	In this article, we could prove that if $X$ is a complete $G_{2}$- (or $Spin(7)$-) manifold with a $d$(linear) $G_{2}$- (or $Spin(7)$-) structure $\phi$ (or $\Om$), then $\mathcal{H}_{(2)}^{k}(X)=\{0\}$ for $k=0,1,2$, See Theorem \ref{T4} and \ref{T5}. We will also show that the structure form could not be $d$(bounded), See Proposition \ref{P6}.
	\begin{remark}
		It well known that if $X$ is a connected, complete, non-compact manifold with nonnegative Ricci curvature, then $\mathcal{H}^{i}_{(2)}(X)=\{0\}$, $i=0,1$. One knows that $G_{2}$-manifold or $Spin(7)$-manifold is Ricci flat, then $\mathcal{H}^{i}_{(2)}(X)=\{0\}$, $i=0,1$. 
	\end{remark}
	Instantons on the higher dimension, proposed in \cite{CDFJ} and studied in \cite{Car,DT,DS,War}, are important in mathematics \cite{DS} and string theory \cite{GSW}. Instantons are important objects in modern field theories. To construct nontrivial solution of instanton equations over a non-compact manifold is very important for high energy physics. It well known that the structures of the cylinders and metric cones over the six-, seven- and eight-dimensional manifolds with structure group $SU(3)$, $G_{2}$ and $Spin(7)$ are inherited from the base manifolds \cite{GLNP}. Constructions of solutions of the instanton equations on cylinders over nearly K\"{a}hler $6$-manifolds and nearly parallel $G_{2}$ manifold were considered in \cite{BILL,HILP,ID,ILPR}. In \cite{GLNP}, they were interested in cone structures constructed over nearly K\"{a}hler $6$-folds $X^{6}$. Its metric cone has $G_{2}$-holonomy if we normalize the nearly K\"{a}hler manifold such that its Eintein constant is $5$. The cylinder over a parallel $G_{2}$-manifold has $Spin(7)$-holonomy. They showed that there was a $G_{2}$-instanton on these $G_{2}$-manifolds which given rise to a $Spin(7)$-instanton in eight dimensions.   
	
	In this article, we observe that if $(X^{n},\w)$ is a complete Riemannian manifold with a $d$(linear)  $k$-form $\w$, then $\int_{X}\a\wedge\w=0$, where $\a$ is a closed $(n-k)$-form in $L^{1}$, See Lemma \ref{L1}. We can prove a vanishing theorem as follows:  if $X$ is a complete $G_{2}$-(or $Spin(7)$-) manifold with a $d$(linear) $G_{2}$- (or $Spin(7)$-) structure $\phi$ (or $\Om$),  the $L^{2}$ solutions of the instanton equation are trivial, See Theorem \ref{T4.3}. In \cite{ILPR} section $4$, the authors confirmed that the standard Yang-Mills functional was infinite on their solutions. The author was  inspired by those results; he proved that the solutions of instantons with square integrable curvature on the cylinder over a compact Riemannian manifold with a real Killing spinor are trivial \cite{Hua}. We observe that the  cylinder $Cyl(X):=(\mathbb{R}\times X, dt^{2}+\rm{g}_{X})$ over a closed Riemanninan manifold $X$ is complete. Combining Corollary \ref{C2}, we can give another way to prove the vanishing theorem in \cite{Hua}. Furthermore, we also prove that if the curvature of the connection satisfies a mild condition, then the instanton is a flat connection, See Theorem \ref{T4.6}.
	\begin{remark}
		The vanishing theorem \ref{T4.3} only means that the nontrivial instantons on a complete Riemannian manifold with a $d$(linear) parallel form must have infinite standard Yang-Mills action. However we cannot catch any information of the topological numbers associated with the instanton solutions, they might even be finite. For example, $\mathbb{R}^{8}$ is a model for the growth conditions required. The well known $Spin(7)$-instanton solution on $\mathbb{R}^{8}$ constructed in S. Fubini and H. Nicolai \cite{FN} has infinite Yang-Mills action but finite topological numbers.
	\end{remark}
	We also consider the Hodge theory on a principal $G$-bundle $E$ over a complete manifold $X$. We denote $H^{p}_{(2)}(X,E)$ by the space of $L^{2}$ harmonic $p$-forms $\La^{p}_{(2)}(X,E)$ respect to the Laplace-Beltrami operator $\De_{A}:=d_{A}d^{\ast}_{A}+d_{A}^{\ast}d_{A}$ (See Definition \ref{D4.5})
	The space $H^{p}_{(2)}(X,E)$ depends on the connection. In this article, we assume that $E$ possesses a flat connection $d_{A}$ which means that $F_{A}=0$, or equivalently, that $E$ is given by a representation $\pi_{X}\rightarrow U(r)$.  Then we would prove that  if $(X,\phi)$ is a complete $G_{2}$-manifold with a $d$(linear) $G_{2}$-structure $\phi$, then $H^{p}_{(2)}(X,E)=0$ unless $p\neq 3,4$, See Theorem \ref{T2}.
	\section{Riemannian manifolds with a parallel differential form}
	In this section, we recall some notations and definitions on differential geometry \cite{Ver}. Let $X$ be a $C^{\infty}$-manifold. We denote by $\La^{\ast}(X)$ the smooth forms on $X$. Given an odd or even from $\a\in\La^{\ast}(X)$, we denote by $\tilde{\a}$ its parity, which is equal to $0$ for even forms, and $1$ for odd forms. An operator $f\in\rm{End}(\La^{\ast}(X))$ preserving parity is called $even$, and one exchanging odd and even forms is odd, $\tilde{f}$ is equal to $0$ for even forms and $1$ for odd ones.
	
	Given a $C^{\infty}$-linear map $\La^{1}(X)\xrightarrow{p}\La^{odd}(X)$ or  $\La^{1}(X)\xrightarrow{p}\La^{even}(X)$, $p$ can be uniquely extended to a $C^{\infty}$-linear derivation $\rho$ on $\La^{\ast}(X)$, using the rule
	\begin{equation*}
	\begin{split}
	&\rho|_{\La^{0}(X)}=0,\\
	&\rho|_{\La^{1}(X)}=p,\\
	&\rho(\a\wedge\b)=\rho(\a)\wedge\b+(-1)^{\tilde{\rho}\tilde{\a}}\a\wedge\rho(\b).\\
	\end{split}
	\end{equation*}
	Verbitsky gave a definition of the structure operator of $(X,\w)$, see \cite{Ver} Definition 2.1.
	\begin{definition}\label{D2.1}
		Let $X$ be a Riemannian manifold equipped with a parallel differential $k$-form $\w$.\ Consider an operator $\underline{C}:\La^{1}(X)\rightarrow\La^{k-1}(X)$ mapping $\a\in\La^{1}(X)$ to $\ast(\ast\w\wedge\a)$. The corresponding differentiation
		$$C:\La^{\ast}(X)\rightarrow\La^{\ast+k-2}(X)$$
		is called the structure operator of $(X,\w)$.
	\end{definition}
	\begin{lemma}\label{L4}
		Let $X$ be a Riemannian manifold equipped with a parallel differential $k$-form $\w$, and  $L_{\w}$ the operator $\a\mapsto\a\wedge\w$. Then
		$$d_{C}=\{L_{\w},d^{\ast}\},$$
		where $d_{C}$ is the supercommutator $\{d,C\}:=dC-(-1)^{\tilde{C}}Cd$.
	\end{lemma}
	We recall some Generalized K\"{a}hler identities which proved by Verbitsky (See \cite{Ver} Proposition 2.5). Here, we give a proof in detail for the reader's convenience.
	\begin{proposition}\label{P1}
		Let $X$ be a Riemannian manifold equipped with a parallel differential $k$-form $\w$, $d_{C}$ the twisted de Rham operator constructed above and $d^{\ast}_{C}$ its Hermitian adjoint. Then,\\
		(i) The following supercommutators vanish:
		$$\{d,d_{C}\}=0,\ \{d,d_{C}^{\ast}\}=0,\ \{d^{\ast},d_{C}\}=0,\ \{d^{\ast},d_{C}^{\ast}\}=0.$$
		(ii) The Laplacian $\De=\{d,d^{\ast}\}$ commutes with $L_{\w}:\a\mapsto\a\wedge\w$ and it adjoint operator, denoted as $\La_{\w}:\La^{i}(X)\rightarrow\La^{i-k}(X)$.
	\end{proposition}
	\begin{proof}
		Let $\de$ be an odd element in a graded Lie superalgebra $A$ satisfying $\{\de,\de\}=0$. Using the graded Jacobi identity, we obtain
		$$\{\de,\{\de,\chi\}\}=-\{\de,\{\de,\chi\}\}+\{\{\de,\de\},\chi\}.$$
		This gives $2\{\de,\{\de,\chi\}\}=0$.
		
		Now,\ $\{d,d_{C}\}=\{d,\{d,d_{C}\}\}=0$ and $\{d^{\ast},d_{C}\}=\{d^{\ast},\{d^{\ast},L_{\w}\}\}=0$ by Lemma \ref{L4}. Taking Hermitian adjoints of these identities, we obtain the other two equations of Proposition \ref{P1} (i).
		
		Now, the graded Jacobi identity implies
		$$[L_{\w},\De]=\{L_{\w},\{d,d^{\ast}\}\}=(-1)^{\tilde{\w}}\{d,\{L_{\w},d^{\ast}\}\}$$
		we use $\{L_{\w},d\}=0$ as $\w$ is closed. This gives
		$$[L_{\w},\De]=(-1)^{\tilde{\w}}\{d,d_{C}\}=0$$
		as Proposition \ref{P1} implies. Taking the Hermitian adjoint, we also obtain $[\La_{\w},\De]=0$.
	\end{proof}
	\begin{corollary}(\cite{Ver} Corollary 2.9)
		Let $(X,\w)$ be a Riemannian manifold equipped with a parallel differential $k$-form $\w$, and $\a$ a harmonic form on $X$. Then $\a\wedge\w$ is harmonic.
	\end{corollary}
	\begin{proof}
		It follows from Proposition \ref{P1} (ii).
	\end{proof}
	\begin{remark}
		If $(X,\w)$ is a $G_{2}$- or $Spin(7)$-manifold, Proposition \ref{P1} gives the Laplacian $\De$ commutes between the operators $L_{\w}$, $\La_{\w}$, $L_{\ast\w}$, $\La_{\ast\w}$.
	\end{remark}
	We begin the proof of Theorem \ref{T1} by recalling some basic facts in Hodge theory. If $X$ is an oriented complete Riemannian manifold, let
	$d^{\ast}$ be the adjoint operator of $d$ acting on the space of $L^{2}$ $k$-forms. Denoted by $\La^{k}_{(2)}(X)$ and $\mathcal{H}^{k}_{(2)}(X)$ the spaces of $L^{2}$ $k$-forms and $L^{2}$ harmonic $k$-forms, respectively. By elliptic regularity and completeness of the manifold, a $k$-form in $\mathcal{H}^{k}_{(2)}(X)$ is smooth, closed and co-closed.
	\begin{definition}
		A  differential form $\w$ on a complete non-compact Riemannian manifold is called $d$(linear) if there exists a differential form $\b$ and a number
		$c>0$ such that 
		\begin{equation*}
		\begin{split}
		&\w=d\b,\\
		&|\w(x)|\leq c,\\
		&|\b(x)|\leq c(1+\rho(x_{0},x)),
		\end{split}
		\end{equation*}
		where $\rho(x_{0},x)$ stands for the Riemannian distance between $x$ and a base point $x_{0}$.
	\end{definition}
	Jost and Zuo's theorem stated that if  a complete K\"{a}hler manifold $X$ with a $d$(linear) K\"{a}hler form $\w$, then the only $L^{2}$-harmonic forms lie in the middle dimension. In \cite{CX}, Cao-Xavier also obtained the same result of Jost-Zuo by another way.
	\begin{theorem}
		Let $(X,\w)$ be a complete K\"{a}hler $n$-manifold with a $d$(linear) K\"{a}hler form. Then all $L^{2}$-harmonic $p$-forms for $p\neq n$ vanish.
	\end{theorem}
	\begin{example}
		Let $(X,\eta,\w)$ be a Sasakian $2n+1$-fold, $\eta$ is a contact $1$-form on $X$. Denoted by $C(X)$ the Riemannian cone of $(X,\rm{g})$. By definition, the Riemannian cone is a product $\mathbb{R}^{+}\times X$, equipped with a metric $dr^{2}+r^{2}\rm{g}$, where $r$ is a unit parameter of $\mathbb{R}^{+}$. Then the Riemannian cone $C(X)$ is a K\"{a}hler-manifold with a K\"{a}hler form $\w$ defined by
		$$\w=r^{2}d\eta+2rdr\wedge\eta,$$
		Since $\Om=d(r^{2}\eta)=d\b$ and $\rho(x_{0},x)=O(r)$, then the Riemaniann cone $C(X)$ is also the model for the growth conditions required.
	\end{example}
	We extend the idea of Cao-Xavier's to the case of Riemannian manifold equipped with a parallel differential form. Then we have
	\begin{theorem}\label{T1}
		Let $(X,\w)$ be a Riemannian manifold equipped with a parallel differential $k$-form $\w$. If $\w$ is also $d$(linear), then for any $\a\in\mathcal{H}^{p}_{(2)}(X)$, we have
		$$\w\wedge\a=0.$$
	\end{theorem}
	\begin{proof}
		Let $\eta:\mathbb{R}\rightarrow\mathbb{R}$ be smooth, $0\leq\eta\leq1$,
		$$
		\eta(t)=\left\{
		\begin{aligned}
		1, &  & t\leq0 \\
		0,  &  & t\geq1
		\end{aligned}
		\right.
		$$
		and consider the compactly supported function
		$$f_{j}(x)=\eta(\rho(x_{0},x)-j),$$
		where $j$ is a positive integer.
		
		Let $\a$ be a harmonic $p$-form in $L^{2}$, and consider the form $\nu=\b\wedge\a$. Observing that $d^{\ast}(\w\wedge\a)=0$ since $\w\wedge\a\in\mathcal{H}_{(2)}^{p+k}(X)$  and noticing that $f_{j}\nu$ has compact support, one has
		$$0=\langle d^{\ast}(\w\wedge\a),f_{j}\nu\rangle_{L^{2}(X)}=\langle\w\wedge\a,d(f_{j}\nu)\rangle_{L^{2}(X)}.$$
		We further note that,\ since $\w=d\b$ and $d\a=0$,
		\begin{equation}\label{E3}
		\begin{split}
		0&=\langle\w\wedge\a, d(f_{j}\nu)\rangle_{L^{2}(X)}\\
		&=\langle\w\wedge\a, f_{j}d\nu\rangle_{L^{2}(X)}+\langle\w\wedge\a, df_{j}\wedge\nu\rangle_{L^{2}(X)}\\
		&=\langle\w\wedge\a, f_{j}\w\wedge\a\rangle_{L^{2}(X)}+\langle\w\wedge\a, df_{j}\wedge\nu\rangle_{L^{2}(X)}\\
		&=\langle\w\wedge\a, f_{j}\w\wedge\a\rangle_{L^{2}(X)}+\langle\w\wedge\a, df_{j}\wedge\b\wedge\a\rangle_{L^{2}(X)}.\\
		\end{split}
		\end{equation}
		Since $0\leq f_{j}\leq 1$ and $\lim_{j\rightarrow\infty}f_{j}(x)(\w\wedge\a)(x)=(\w\wedge\a)(x)$,it follows from the dominated convergence theorem that
		\begin{equation}\label{E4}
		\lim_{j\rightarrow\infty}\langle\w\wedge\a, f_{j}\w\wedge\a\rangle_{L^{2}(X)}=\|\w\wedge\a\|^{2}_{L^{2}(X)}.
		\end{equation}
		Since $\w$ is bounded,\ $supp(df_{j})\subset B_{j+1}\backslash B_{j}$ and $|\b(x)|=O(\rho(x_{0},x))$, one obtains
		\begin{equation}\label{E2}
		|\langle\w\wedge\a,df_{j}\wedge\b\wedge\a\rangle_{L^{2}(X)}|\leq (j+1)C\int_{B_{j+1}\backslash B_{j}}|\a(x)|^{2}dx,
		\end{equation}
		where $C$ is a constant independent of $j$.
		
		We claim that there exists a subsequence $\{j_{i}\}_{i\geq1}$ such that
		\begin{equation}\label{E1}
		\lim_{i\rightarrow\infty}(j_{i}+1)\int_{B_{j_{i}+1}\backslash B_{j_{i}}}|\a(x)|^{2}dx=0.
		\end{equation}
		If not, there would exist a positive constant $a$ such that
		$$\lim_{i\rightarrow\infty}(j_{i}+1)\int_{B_{j_{i}+1}\backslash B_{j_{i}}}|\a(x)|^{2}dx\geq a>0,\ j\geq1.$$
		This inequality implies
		\begin{equation}\nonumber
		\begin{split}
		\int_{X}|\a(x)|^{2}dx&=\sum_{j=0}^{\infty}\int_{B_{j+1}\backslash B_{j}}|\a(x)|^{2}dx\\
		&\geq a\sum_{j=0}^{\infty}\frac{1}{j+1}=+\infty\\
		\end{split}
		\end{equation}
		a contradiction to the assumption $\int_{X}|\a(x)|^{2}dx<\infty$. Hence, there exists a subsequence $\{j_{i}\}_{i\geq1}$ for which (\ref{E1}) holds. Using (\ref{E2}) and (\ref{E1}), one obtains
		\begin{equation}\label{E5}
		\lim_{i\rightarrow\infty}\langle\w\wedge\a, df_{j}\wedge\b\wedge\a\rangle_{L^{2}(X)}=0
		\end{equation}
		It now follows from (\ref{E3}), (\ref{E4}) and (\ref{E5}) that $\w\wedge\a=0$.
	\end{proof}
	\begin{remark}
		There are many complete manifolds with a $d$(linear) parallel differential form. If $X$ is a complete simply-connected manifold of non-positive sectional curvature and $\w$ is a parallel differential $k$-form on $X$, then the Theorem 1.1 on \cite{CX} states that $\w$ is $d$(linear).
	\end{remark}
	\begin{corollary}\label{C1}
		Let $(X,\w)$ be a Riemannian manifold equipped with a nonzero parallel differential $k$-form $\w$. If $\w$ is also $d$(linear), then $\mathcal{H}^{0}_{(2)}(X)=0$.
	\end{corollary}
	\begin{proof}
		We denote by $f$  a $L^{2}$-harmonic function on $X$. Then following Theorem \ref{T1}, $f\w=0$. Since $\w$ is nonzero all over $X$, it follows that $f$ vanish.
	\end{proof}
	As we derive estimates in this section (and also following sections), there will be many constants which appear. Sometimes, we will take care to bound the size of these constants, but we will also use the following notation whenever the value of the constants are unimportant. We write $\a\lesssim\b$ to mean that $\a\leq C\b$ for some positive constant $C$ independent of certain parameters on which $\a$ and $\b$ depend. The parameters on which $C$ is independent will be clear or specified at each occurrence. We also use $\b\lesssim\a$ and $\a\approx\b$ analogously.
	
	If we suppose the parallel $k$-form $\w$ is $d$(bounded), following the idea of Gromov \cite{Gro}, we can give a lower bound on the spectrum of the Laplace operator $\De$ on $\La^{(0)}_{(2)}$. 
	\begin{proposition}\label{P10}
		Let $(X,\w)$ be a Riemannian $n$-manifold equipped with a parallel, nonzero, differential $k$-form $\w$. If $\w$ is $d$(bounded), i.e, there exists a bounded $k-1$-form $\theta$ such that $\w=d\theta$, then any $\a\in\La^{0}_{(2)}(X)$ satisfies the inequality
		$$\|\a\|^{2}_{L^{2}(X)}\leq C\|\theta\|^{2}_{L^{\infty}(X)}\langle\De\a,\a\rangle_{L^{2}(X)},$$
		where $C=C(X,n)$ is a positive constant.
	\end{proposition}
	\begin{proof}
		Since $\w$ is a parallel differential form, then $\na|\w|^{2}=0$, i.e. $|\w|=constant$. Denoted $u\in\La^{0}(X)$, we observe that:
		$$|u\wedge\w|^{2}=\ast\big{(}(u\wedge\w)\wedge\ast(u\wedge\w)\big{)}=constant|u|^{2},$$
		and
		$$\De(u\wedge\w)\wedge\ast(u\wedge\w)=(\De u\wedge\w)\wedge\ast(u\wedge\w)=constant(\De u\wedge\ast u).$$
		These imply that
		$$\|u\|_{L^{2}(X)}=constant\|u\wedge\w\|_{L^{2}(X)},\ \langle\De(u\wedge\w),u\wedge\w\rangle_{L^{2}(X)}=constant\langle\De u,u\rangle_{L^{2}(X)}.$$
		Now, we write $\b=\a\wedge\w=d\eta-\tilde{\a}$, for $\eta=\a\wedge\theta$ and $\tilde{\a}=d\a\wedge\theta$ and observe that
		$$\|\eta\|_{L^{2}(X)}\lesssim\|\theta\|_{L^{\infty}(X)}\|\a\|_{L^{2}(X)}.$$
		Next, since
		\begin{equation}\nonumber
		\begin{split}
		\|\tilde{\a}\|_{L^{2}(X)}&\lesssim\|d\a\|_{L^{2}(X)}\|\theta\|_{L^{\infty}(X)}\\
		&\lesssim\langle\De\a,\a\rangle_{L^{2}(X)}^{1/2}\|\theta\|_{L^{\infty}(X)},\\
		\end{split}
		\end{equation}
		we have
		\begin{equation}\nonumber
		\begin{split}
		\|\b\|^{2}_{L^{2}(X)}&\leq|\langle\b,d\eta\rangle_{L^{2}(X)}|+|\langle\b,\tilde{\a}\rangle_{L^{2}(X)}|\\
		&\leq|\langle d^{\ast}\b,\eta\rangle_{L^{2}(X)}|+|\langle\b,\tilde{\a}\rangle_{L^{2}(X)}|\\
		&\lesssim\langle\De\b,\b\rangle_{L^{2}(X)}^{1/2}\|\theta\|_{L^{\infty}(X)}\|\b\|_{L^{2}(X)}
		+\|\b\|_{L^{2}(X)}\|d\a\|_{L^{2}(X)}\|\theta\|_{L^{\infty}(X)}\\
		&\lesssim\langle\De\a,\a\rangle_{L^{2}(X)}^{1/2}\|\theta\|_{L^{\infty}(X)}\|\b\|_{L^{2}(X)}.\\
		\end{split}
		\end{equation}
		This yields the desired estimate
		$$
		\|\a\|^{2}_{L^{2}(X)}\lesssim\|\b\|^{2}_{L^{2}(X)}
		\lesssim\|\theta\|^{2}_{L^{\infty}(X)}\langle\De\a,\a\rangle_{L^{2}(X)}.$$
		We complete this proof.
	\end{proof}
	In \cite{CY}, Cheng and Yau proved that the first eigenvaule of Laplace operator $\De$ is zero on a complete Ricci-flat manifold. Hence one can easily see the $G_{2}$- or $Spin(7)$-structure could not be $d$(bounded) since the Proposition \ref{P10} states that the first eigenvaule is nonzero if the structure form  is $d$(bounded).
	\begin{proposition}\label{P6}
		If $\phi$ (or $\Om$) is the $G_{2}$- (or $Spin(7)$-) structure from over a complete, non-compact $G_{2}$- (or $Spin(7)$-) manifold, then $\phi$ (or $\Om$)  could be not $d$(bounded).
	\end{proposition}
	\section{Special holonomy manifolds}
	\subsection{$G_{2}$-manifolds}
	\begin{definition}
		A $G_{2}$-manifold is a $7$-manifold $X$ equipped with a torsion-free $G_{2}$-structure $\phi$, that is $$\na_{\rm{g}_{\phi}}\phi=0,$$
		where $\rm{g}_{\phi}$ is the metric induced by $\phi$.
	\end{definition}
	Under the action of $G_{2}$, the space $\La^{2}(X)$ splits into irreducible representations, as follows:
	\begin{equation}\label{E6}
	\La^{2}(X)=\La^{2}_{7}(X)\oplus\La_{14}^{2}(X).
	\end{equation}
	where $\La^{i}_{j}$ is an irreducible $G_{2}$-representation of dimension $j$.
	These summands can be characterized as follows:
	\begin{equation}\nonumber
	\begin{split}
	&\La^{2}_{7}(X)=\{\a\in\La^{2}(X)\mid\ast(\a\wedge\phi)=2\a\},\\
	&\La^{2}_{14}(X)=\{\a\in\La^{2}(X)\mid\ast(\a\wedge\phi)=-\a\}=\{\a\in\La^{2}(X)\mid\a\wedge\ast\phi=0\}.\\
	\end{split}
	\end{equation}
	From the construction, it is clear that the splitting (\ref{E6}) can be obtained via the operator $L_{\phi}$, $\La_{\phi}$, $L_{\ast\phi}$, $\La_{\ast\phi}$. By Proposition \ref{P1} these operators commute with the Laplacian. Therefore, the harmonic forms also split:
	\begin{equation}\nonumber
	\begin{split}
	&\mathcal{H}^{2}_{(2)}(X)=\mathcal{H}^{2}_{7;(2)}(X)\oplus\mathcal{H}^{2}_{14;(2)}(X).\\
	\end{split}
	\end{equation}
	\begin{example}
		Let $(X,\w,\Om)$ be a nearly K\"{a}hler $6$-fold, see \cite{Ver2005,Ver2011}. There is a $(3,0)$-form $\Om$ with $|\Om|=1$, and
		$$d\w=3\la Re\Om,\ dIm\Om=-2\la\w^{2},$$
		where $\la$ is a nonzero real constant. For simply, we choose $\la=1$. Denoted by $C(X)$ the Riemannian cone of $(X,\rm{g})$. The Riemannian cone $\big{(}C(X),dr^{2}+r^{2}\rm{g}\big{)}$ is a $G_{2}$-manifold with torsion-free $G_{2}$-structure $\phi$ defined by
		$$\phi:=r^{2}\w\wedge dr+r^{3}Re\Om.$$
		Since $\phi=d(\frac{1}{3}r^{3}\w)=d\b$ and $\rho(x_{0},x)=O(r)$, then the Riemaniann cone $C(X)$ is also the model for the growth conditions required.
	\end{example}
	We will  show that the map $L_{\phi}:\La^{p}\rightarrow\La^{p+3}$ on the complete $G_{2}$-manifold is  injective for $p=0,1,2$ .
	\begin{lemma}\label{L2}
		Let $(X,\phi)$ be a complete $G_{2}$-manifold, for any $\a\in\La^{k}(X)$, $k=0,1,2$, satisfies the inequalities
		$$\|\a\|_{L^{2}(X)}\approx\|\a\wedge\phi\|_{L^{2}(X)},$$
		$$\langle\De\a,\a\rangle_{L^{2}(X)}\approx\langle\De(\a\wedge\phi),\a\wedge\phi\rangle_{L^{2}(X)}.$$
	\end{lemma}
	\begin{proof}
		Let $\a,\b\in\La^{0}(X)$,\ we observe that:
		$$(\a\wedge\phi)\wedge\ast(\b\wedge\phi)=7\a\b\ast1.$$
		We take $\b=\a$, then
		$$\|\a\|^{2}_{L^{2}(X)}=\frac{1}{7}\|\a\wedge\phi\|^{2}_{L^{2}(X)},\ \langle\De\a,\a\rangle_{L^{2}(X)}=\frac{1}{7}\langle\De(\a\wedge\phi),\a\wedge\phi\rangle_{L^{2}(X)}.$$
		Let $\a,\b\in\La^{1}(X)$, we also observe that:
		$$\ast(\a\wedge\phi)\wedge(\b\wedge\phi)=4\ast\a\wedge\b,$$
		here we use the fact $\ast(\a\wedge\phi)\wedge\phi=-4\ast\a$ (See \cite{Bry2}). We take $\b=\a$, then
		$$\|\a\|^{2}_{L^{2}(X)}=\frac{1}{4}\|\a\wedge\phi\|^{2}_{L^{2}(X)},\ \langle\De\a,\a\rangle_{L^{2}(X)}=\frac{1}{4}\langle\De(\a\wedge\phi),\a\wedge\phi\rangle_{L^{2}(X)}.$$
		Let $\a\in\La^{2}(X)$, we can write $\a=\a^{7}+\a^{14}$, then  $\a\wedge\phi=2\ast\a^{7}-\ast\a^{14}$. Hence
		\begin{equation}\nonumber
		\|\a\wedge\phi\|^{2}_{L^{2}(X)}=4\|\a^{7}\|^{2}_{L^{2}(X)}+\|\a^{14}\|^{2}_{L^{2}(X)}\approx\|\a\|^{2}_{L^{2}(X)}.
		\end{equation}
		Since $[\De,L_{\phi}]=0$, we have $\De(\a\wedge\phi)=\De\a\wedge\phi=\ast\De(2\a^{7}-\a^{14})$.\ Then
		\begin{equation}\nonumber
		\begin{split}
		\langle\De(\a\wedge\phi),\a\wedge\phi\rangle_{L^{2}(X)}&=\langle\ast\De(2\a^{7}-\a^{14}),\ast(2\a^{7}-\a^{14})\rangle_{L^{2}(X)}\\
		&=4\langle\De\a^{7},\a^{7}\rangle_{L^{2}(X)}+\langle\De\a^{14},\a^{14}\rangle_{L^{2}(X)}\\
		&\approx\langle\De\a,\a\rangle_{L^{2}(X)}.\\
		\end{split}
		\end{equation}
	\end{proof}
	\begin{theorem}\label{T4}
		Let $(X,\phi)$ be a complete $G_{2}$-manifold with a $d$(linear) $G_{2}$-structure. Then, $\mathcal{H}^{k}_{(2)}(X)=\{0\}$ for $k=0,1,2$.
	\end{theorem}
	\begin{proof}
		We denote $\a$ by a harmonic $p$-form $\a$. Following the hypothesis of the structure form $\phi$ , we have $\a\wedge\phi=0$ (See Lemma \ref{L2}). Since $L_{\phi}:\La^{p}(X)\rightarrow\La^{p+3}(X)$ is injective for $p=0,1,2$ (See Lemma \ref{L2}), we have $\a\equiv0$.
	\end{proof}
	If we suppose that the $G_{2}$-structure $4$-form $\ast\phi$ is $d$(linear), we would also prove another vanishing theorem.
	\begin{theorem}
		Let $(X,\phi)$ be a complete $G_{2}$-manifold. If $\ast\phi$ is a $d$(linear) form, then $\mathcal{H}^{2}_{(2)}(X)=\{0\}$.
	\end{theorem}
	\begin{proof}
		We denote $\a$ by a harmonic $L^{2}$-form of degree $2$. We also consider the form $\a\wedge\ast\phi$, following Theorem \ref{T1}, $\a\wedge\ast\phi=0$, i.e., $\a+\ast(\a\wedge\phi)=0$. On this time, the map $L_{\ast\phi}:\La^{2}(X)\rightarrow\La^{6}(X)$ is not injective. But $tr(\a\wedge\a)$ is closed $L^{1}$ form on $X$, following  Lemma \ref{L1}, $\|\a\|^{2}_{L^{2}(X)}=-\int_{X}tr(\a\wedge\a\wedge\phi)=0$, i.e., $\a=0$.
	\end{proof}
	\begin{example}
		Let $(X,\eta,\w)$ be a Sasakian-Einstein $5$-fold, $\eta$ is a contact $1$-form on $X$. The metric cone $C(X)$ is a Calabi-Yau manifold. There are  K\"{a}hler form $\w=d(\frac{1}{2}r^{2}\eta)$ and volume form $\Om\in\La^{3,0}(X)$ which satisfies $\na\Om=0$. Denoted by $Cyl(C(X))$ the cylinder over the Calabi-Yau manifold $C(X)$. We can use the $\w,\Om$ on the base $C(X)$ to define a $G_{2}$-structure:
		$$\phi=dt\wedge\w+Im\Om$$
		and
		$$\ast\phi=\frac{1}{2}\w^{2}+dt\wedge\ Re\Om.$$
		where the metric on $Cyl(C(X))$ is $dt^{2}+dr^{2}+r^{2}\rm{g}_{X}$. 
		Since $\ast\phi=d(\w\wedge\frac{1}{2}r^{2}\eta+tRe\Om)$ and $\rho(x_{0},x)=O((r^{2}+t^{2})^{1/2})$,\ then the $G_{2}$-manifold $Cyl(C(X))$ has a linear growth parallel form $\ast\phi$.
	\end{example}
	\subsection{$Spin(7)$-manifolds}
	\begin{definition}
		A $Spin(7)$-manifold is a $8$-manifold $X$ equipped with a torsion-free $Spin(7)$-structure $\Om$, that is $$\na_{\rm{g}_{\Om}}\Om=0,$$
		where $\rm{g}_{\Om}$ is the metric induced by $\Om$.
	\end{definition}
	Under the action of $Spin(7)$, the space $\La^{2}(X)$ splits into irreducible representations, as follows:
	\begin{equation}\label{E13}
	\La^{2}(X)=\La^{2}_{7}(X)\oplus\La^{2}_{21}(X).
	\end{equation}
	These summands can be characterized as follows:
	\begin{equation}\nonumber
	\begin{split}
	&\La^{2}_{7}(X)=\{\a\in\La^{2}(X)\mid\ast(\a\wedge\Om)=3\a\},\\
	&\La^{2}_{21}(X)=\{\a\in\La^{2}(X)\mid\ast(\a\wedge\Om)=-\a\}.\\
	\end{split}
	\end{equation}
	From the construction, it is clear that the splitting (\ref{E13}) can be obtained via the operator $L_{\Om}$, $\La_{\Om}$. By Proposition \ref{P1} these operators commute with the Laplacian. Therefore, the harmonic forms also split:
	\begin{equation}\nonumber
	\begin{split}
	&\mathcal{H}^{2}_{(2)}(X)=\mathcal{H}^{2}_{7;(2)}(X)\oplus\mathcal{H}^{2}_{21;(2)}(X).\\
	\end{split}
	\end{equation}
	\begin{example} 
		Let $(X,\phi)$ be a nearly parallel $G_{2}$-manifold (See \cite{Iva}). There is a $3$-form $\phi$ with $|\phi|^{2}=7$ such that
		$$d\phi=4\ast\phi.$$
		Then the Riemannian cone $\big{(}C(X),dr^{2}+r^{2}\rm{g}\big{)}$ is a $Spin(7)$-manifold with  $Spin(7)$-structure $\Om$ defined by
		$$\Om:=r^{3}dr\wedge\phi+r^{4}\ast\phi.$$
		Since $\phi=d(\frac{1}{4}r^{4}\phi)=d\b$ and $\rho(x_{0},x)=O(r)$, the Riemaniann cone $C(X)$ is also the model for the growth conditions required.
	\end{example}
	We will also show that the map $L_{\Om}:\La^{p}\rightarrow\La^{p+4}$ on the complete $Spin(7)$-manifold is injective for $p=0,1,2$.
	\begin{lemma}\label{L3}
		Let $(X,\Om)$ be a complete $Spin(7)$-manifold, for any $\a\in\La^{k}(X)$, $k=0,1,2$, satisfies the inequalities
		$$\|\a\|_{L^{2}(X)}\approx\|\a\wedge\Om\|_{L^{2}(X)},$$
		$$\langle\De\a,\a\rangle_{L^{2}(X)}\approx\langle\De(\a\wedge\Om),\a\wedge\Om\rangle_{L^{2}(X)}.$$
	\end{lemma}
	\begin{proof}
		Let $\a,\b\in\La^{0}(X)$,\ we observe that:
		$$(\a\wedge\Om)\wedge\ast(\b\wedge\Om)=14\a\b\ast1,$$
		then
		$$\|\a\|^{2}_{L^{2}(X)}=\frac{1}{14}\|\a\wedge\Om\|_{L^{2}(X)},\ \langle\De\a,\a\rangle_{L^{2}(X)}=\frac{1}{12}\langle\De(\a\wedge\Om),\a\wedge\Om\rangle_{L^{2}(X)}.$$
		Let $\a,\b\in\La^{1}(X)$, we also observe that:
		$$\ast(\a\wedge\Om)\wedge(\b\wedge\Om)=4\ast\a\wedge\b,$$
		here we use the fact $\ast(\a\wedge\Om)\wedge\Om=4\ast\a$. We take $\b=\a$, then
		$$\|\a\|^{2}_{L^{2}(X)}=\frac{1}{4}\|\a\wedge\Om\|^{2}_{L^{2}(X)},\ \langle\De\a,\a\rangle_{L^{2}(X)}=\frac{1}{4}\langle\De(\a\wedge\Om),\a\wedge\Om\rangle_{L^{2}(X)}.$$
		Let $\a\in\La^{2}(X)$, we write $\a=\a^{7}+\a^{21}$, then  $\a\wedge\Om=3\ast\a^{7}-\ast\a^{21}$. Hence
		\begin{equation}\nonumber
		\|\a\wedge\Om\|^{2}_{L^{2}(X)}=9\|\a^{7}\|^{2}_{L^{2}(X)}+\|\a^{21}\|^{2}_{L^{2}(X)}\approx\|\a\|^{2}_{L^{2}(X)}.
		\end{equation}
		Since $[\De,L_{\Om}]=0$, we have $\De(\a\wedge\Om)=\De\a\wedge\Om=\ast\De(3\a^{7}-\a^{21})$. Then
		\begin{equation}\nonumber
		\begin{split}
		\langle\De(\a\wedge\Om),\a\wedge\Om\rangle_{L^{2}(X)}&=\langle\ast\De(3\a^{7}-\a^{21}),\ast(3\a^{7}-\a^{21})\rangle_{L^{2}(X)}\\
		&=9\langle\De\a^{7},\a^{7}\rangle_{L^{2}(X)}+\langle\De\a^{21},\a^{21}\rangle_{L^{2}(X)}\\
		&\approx\langle\De\a,\a\rangle_{L^{2}(X)}.\\
		\end{split}
		\end{equation}
	\end{proof}
	\begin{theorem}\label{T5}
		Let $(X,\Om)$ be a complete $Spin(7)$-manifold with a $d$(linear) $Spin(7)$-structure. Then $\mathcal{H}^{k}_{(2)}(X)=\{0\}$ for $k=0,1,2$.
	\end{theorem}
	\begin{proof}
		We denote $\a$ by a harmonic $p$-form $\a$. Following the hypothesis of the structure form $\Om$ , we have $\a\wedge\Om=0$ (See Lemma \ref{L2}). Since $L_{\Om}:\La^{p}(X)\rightarrow\La^{p+4}(X)$ is injective for $p=0,1,2$ (See Lemma \ref{L3}), we have $\a\equiv0$.	
	\end{proof}
	\section{Gauge theory}
	\subsection{Instantons}
	We consider the instanton equation on the geometries discussed in the previous section. Let $E$ be a principal $G$-bundle over a complete Riemannian manifold $X$, with dimension $n$ and $A$ be a connection on bundle $E$ over $X$. The instanton equation on $X$ can be introduced as follows. Assume there is a $4$-form $Q$ on $X$. Then a $(n-4)$-form $\ast{Q}$ exists, where $\ast$ is the Hodge operator on $X$. A connection $A$ is called an anti-self-dual instanton, when it satisfies the instanton equation
	\begin{equation}\label{1.1}
	\ast F_{A}+\ast{Q}\wedge F_{A}=0
	\end{equation}
	When $n>4$,\ these equations can be defined on the manifold $X$ with a special holonomy group, i.e., the holonomy group $\rm{Hol}(X)$ of the Levi-Civita connection on the tangent bundle $TX$ is a subgroup of the group $SO(n)$. Each solution of equation (\ref{1.1}) satisfies the Yang-Mills equation. The instanton equation (\ref{1.1}) is also well-defined on a manifold $X$ with non-integrable $G$-structures, but equation (\ref{1.1}) implies the Yang-Mills equation will have torsion. For our purposes, $X$ is a $G_{2}$-manifold and $\ast Q$ is the $G_{2}$-structure $3$-form or $X$ is a $Spin(7)$-manifold and $\ast Q$ is the $Spin(7)$-structure $4$-from. At first, we prove a useful lemma 
	\begin{lemma}\label{L1}
		Let $(X^{n},\w)$ be a complete Riemannian $n$-manifold with a $d$(linear) $k$-form $\w$. Suppose that $\w$ is bounded. If $\a$ is a closed $L^{1}$ form of degree $n-k$, then
		$$\int_{X}\a\wedge\w=0.$$
	\end{lemma}
	\begin{proof}
		Let $\a$ be a closed $(n-k)$-form in $L^{1}$, and noticing that $f_{j}$ is as the cutoff function in the proof of Theorem \ref{T1}, one has
		\begin{equation}\label{E7}
		\begin{split}
		\langle f_{j}\a,\ast\w\rangle_{L^{2}(X)}
		&=\langle f_{j}\a,\ast d\b\rangle_{L^{2}(X)}\\
		&=(\pm)\langle d(f_{j}\a),\ast\b\rangle_{L^{2}(X)}\\
		&=(\pm)\big{(}\langle df_{j}\wedge\a,\ast\b\rangle_{L^{2}(X)}+\langle f_{j}d\a,\ast\b\rangle_{L^{2}(X)}\big{)}\\
		&=(\pm)\langle df_{j}\wedge\a,\ast\b\rangle_{L^{2}(X)}.\\
		\end{split}
		\end{equation}
		Since $0\leq f_{j}\leq 1$ and $\lim_{j\rightarrow\infty}f_{j}(x)\a(x)=\a(x)$, it follows from
		the dominated convergence theorem that
		\begin{equation}\label{E10}
		\lim_{j\rightarrow\infty}\langle f_{j}\a,\ast\w\rangle_{L^{2}(X)}=\int_{X}\a\wedge\w.
		\end{equation}
		Since $\w$ is bounded,\ $supp(df_{j})\subset B_{j+1}\backslash B_{j}$ and $|\b(x)|=O(\rho(x_{0},x))$, one obtains
		\begin{equation}\label{E8}
		|\langle df_{j}\wedge\a,\ast\b\rangle_{L^{2}(X)}|\leq (j+1)C\int_{B_{j+1}\backslash B_{j}}|\a(x)|dx,
		\end{equation}
		where $C$ is a constant independent of $j$. Using the similar proof in Theorem \ref{T1}, we can proof that there exists a subsequence $\{j_{i}\}_{i\geq1}$ such that
		\begin{equation}\label{E9}
		\lim_{i\rightarrow\infty}(j_{i}+1)C\int_{B_{j_{i}+1}\backslash B_{j_{i}}}|\a(x)|dx=0.
		\end{equation}
		It now follows from (\ref{E7}), (\ref{E10}) and (\ref{E9}) that $\int_{X}\a\wedge\w=0$.
	\end{proof}
	\begin{corollary}\label{C2}
		Let $(X^{n},\w)$ be a complete Riemanniann manifold with a $d$(linear) $(n-4)$-form $\w$, $E$ be a principal $G$-bundle on $X$ and $A$ be a smooth connection on $E$. Suppose that $\w$ is bounded. If the curvature $F_{A}$ is in $L^{2}$, then
		$$\int_{X}tr(F_{A}\wedge F_{A})\wedge\w=0.$$
	\end{corollary}
	\begin{proof}
		From the Bianchi identity $d_{A}F_{A}=0$, we have 
		$$dtr(F_{A}\wedge F_{A})=tr(d_{A}(F_{A}\wedge F_{A}))=0.$$
		Thus $dtr(F_{A}\wedge F_{A})$ is an $L^{1}$ closed form. Following Lemma \ref{L1}, we can complete the proof of this Corollary.
	\end{proof}
	We then have a vanishing theorem on the $G_{2}$- (or $Spin(7)$-) instantons  over a complete manifold with $d$(linear) structure form.
	\begin{theorem}\label{T4.3}
		Let $X$ be a complete $G_{2}$-(or $Spin(7)$-) manifold with a $d$(linear) $G_{2}$- (or $Spin(7)$-) structure $\phi$ (or $\Om$), $E$ be a $G$-bundle on $X$ and $A$ be a smooth connection on $E$. If the connection $A$ is a $G_{2}$- (or $Spin(7)$-) instanton with square integrable curvature $F_{A}$, then $A$ is a flat connection.
	\end{theorem}
	\begin{proof}
		By the hypothesis of the connection $A$, the Yang-Mills  energy functional on a complete $G_{2}$-manifold is
		$$YM(A)=\int_{X}tr(F_{A}\wedge F_{A})\wedge\phi.$$ 
		Following Corollary \ref{C2}, we obtain $YM(A)=0$, i.e. $F_{A}\equiv0$.  
	\end{proof}

	Let $(X,\rm{g})$ be a real Killing spinor compact manifold of dimension $n$, i.e., there are $3$-form $P$ and $4$-form $Q$ which satisfy
	$$dP=4Q,\ d\ast_{X}Q=(n-3)\ast_{X}P,$$
	where $\ast_{X}$ is the Hodge star operator on $X$. For $n>3$, the Chern-Simons functional can then be written as
	\begin{equation}\label{E4.6}
	CS(A)=-\frac{1}{2(n-3)}\int_{X}tr(F_{A}\wedge F_{A})\wedge\ast_{X}Q,
	\end{equation}
	which is gauge-invariant. We consider the cylinder $Cyl(X):=\mathbb{R}\times X$ over $X$. Then we can define a $4$-form $\Om$ on $Cyl(X)$ as
	$$\Om=dt\wedge P+Q,$$
	with $t$ the linear coordinate on $\mathbb{R}$. Let $\mathcal{A}$ be a gauge field on $Cyl(X)$ with the property that $dt\lrcorner A$, which is simply a choice of gauge. The instanton equation on the cylinder splits into the two equations
	\begin{equation}\label{E4.13}
	\begin{split}
	&\ast_{X}\frac{\pa A}{\pa t}=-\ast_{X}P\wedge F_{A},\\
	&\ast_{X}F_{A}=\ast_{X}Q\wedge F_{A}-\frac{\pa A}{\pa t}\wedge\ast_{X}P.\\
	\end{split}
	\end{equation}
	The  gradient flow of Chern-Simons functional (\ref{E4.6}) is equivalent to the first of equations (\ref{E4.13}). We denote $\ast$ by the Hodge star operator on $Cyl(X)$, $D$ by the exterior derivative on $T^{\ast}(Cyl(X))$. We also denote $\tilde{P}=dt\wedge P$, $\tilde{Q}=dt\wedge Q$. Then the forms $\tilde{P}$, $\tilde{Q}$ satisfy 
	$$\ast\tilde{P}=\ast_{X}P,\  \ast\tilde{Q}=\ast_{X}Q,$$
	and
	$$D\tilde{P}=4\tilde{Q},\ D\ast\tilde{Q}=(n-3)\ast\tilde{P}.$$
	The Yang-Mills energy function is 
	\begin{equation}\nonumber
	\begin{split}
	YM(\mathcal{A}):&=\|F_{\mathcal{A}}\|^{2}_{L^{2}(Cyl(X))}=-\int_{\mathbb{R}\times X}tr(F_{\mathcal{A}}\wedge F_{\mathcal{A}})\wedge\ast\Om\\
	&=-\int_{\mathbb{R}\times X}tr(F_{\mathcal{A}}^{2})\wedge\ast\tilde{P}-\int_{\mathbb{R}\times X}tr(F_{\mathcal{A}}^{2})\wedge\ast_{X}Q\wedge dt.\\
	\end{split}
	\end{equation}
	We observe that 
	$$-\int_{\mathbb{R}\times X}tr(F_{\mathcal{A}}^{2})\wedge\ast\tilde{P}=-\frac{1}{n-3}\int_{\mathbb{R}\times X}tr(F_{\mathcal{A}}^{2})\wedge D\ast\tilde{Q}.$$ 
	We also observe that
	$$F_{\mathcal{A}}=\frac{\pa A}{\pa t}\wedge dt+F_{A}$$
	and
	$$-tr(F_{\mathcal{A}}^{2})\wedge\ast\tilde{P}=-2tr(\frac{\pa A}{\pa t}\wedge dt\wedge F_{A})\wedge\ast_{X}P=2|\frac{\pa A}{\pa t}|^{2}dt\wedge dvol,$$
	here we use the first equation on (\ref{E4.13}). Thus
	\begin{equation}\label{E4.12}
	-\int_{\mathbb{R}\times X}tr(F_{\mathcal{A}}^{2})\wedge\ast\tilde{P}=2\int_{\mathbb{R}\times X}|\frac{\pa A}{\pa t}|^{2}dt\wedge dvol.
	\end{equation}
	In \cite{Hua}, the author proved a vanishing theorem as follows:
	\begin{theorem}( \cite{Hua} Theorem 1.2)
		If the connection $\mathcal{A}$ is a solution of $\Om$-instanton equation with square integrable curvature $F_{\mathcal{A}}$ over $Cyl(X)$, where $X$ is a compact Real Killing spinor manifold, then $\mathcal{A}$ is flat.
	\end{theorem}
	\begin{proof}
		If $F_{\mathcal{A}}$ is in $L^{2}(Cyl(X))$, following the Corollary \ref{C2}, then it implies that
		$$\int_{\mathbb{R}\times X }tr(F_{\mathcal{A}}^{2}\wedge\ast\tilde{P})=0.$$ Combining Equation (\ref{E4.12})  gives $\frac{\pa A}{\pa t}=0$, i.e., the connection $\mathcal{A}$ is not dependence on parameter $t$. Thus the Yang-Mills functional $YM(\mathcal{A})=\int_{\mathbb{R}}dt\int_{X}|F_{A}|^{2}dvol$ is finite if only if $F_{A}=0$. 	
	\end{proof}
	In this article, we will show that if the standard Yang-Mills functional on $Cyl(X)$ satisfies some mild conditions, then the solution of $\Om$-instanton equation is trivial. One also can see Section 4 on \cite{Hua2}. We define the energy density $\rho(\mathcal{A})$ by
	$$\rho(\mathcal{A}):=\lim_{T\rightarrow\infty}\frac{1}{2T}\int_{(-T,T)\times X}|F_{\mathcal{A}}|^{2}dvol_{\rm{g}}dt.$$
	\begin{lemma}\label{L4.5}(\cite{Hua2} Lemma 4.2 )
		Let $X$ be a complete manifold of dimension $n$ with a $d$(bounded) $k$-form $\w$, i.e.. there exist a $(k-1)$-form $\theta$ such that $\w=d\theta$, $\a$ be a closed from of degree $n-k$. If $\a$ satisfies
		\begin{equation}\label{E4.7}
		\lim_{r\rightarrow\infty}\frac{1}{r}\int_{B_{r}(x_{0})}|\a|dvol=0,
		\end{equation} 
		where $x_{0}$ is a point on $X$, $B_{r}(x_{0})$ is a geodesic ball, then there exists a sequence $\{j_{i}\}_{i\geq1}$ such that
		$$\lim_{i\rightarrow\infty}\int_{B_{j_{i}}(x_{0})}\a\wedge\w=0.$$
	\end{lemma}
	\begin{proof}
		We denote $f_{j}$ by the cutoff function in the proof of Theorem \ref{T1}.  We consider the form $\b:=\a\wedge\w=d(\a\wedge\theta)$. We have $f_{j}\b=d(f_{j}\a\wedge\theta)-df_{j}\wedge(\a\wedge\theta)$. By Stokes formula, we obtain
		\begin{equation*}
		|\int_{X}f_{j}\b|=|\int_{X}df_{j}\wedge(\a\wedge\theta)|\lesssim\int_{B_{j+1}\backslash B_{j}}|\a|
		\end{equation*}
		and
		\begin{equation*}
		|\int_{B_{j}}\b|\leq|\int_{X}f_{j}\b|+\int_{B_{j+1}\backslash B_{j}}|\b|\lesssim|\int_{X}f_{j}\b|+\int_{B_{j+1}\backslash B_{j}}|\a|
		\end{equation*}
		Thus
		\begin{equation}\label{E4.9}
		|\int_{B_{j}}\b|\lesssim\int_{B_{j+1}\backslash B_{j}}|\a|.
		\end{equation}
		By the hypothesis (\ref{E4.7}), there exists a subsequence $\{j_{i}\}_{i\geq1}$ such that
		\begin{equation}\label{E4.10}
		\lim_{i\rightarrow\infty}\int_{B_{j_{i}+1}\backslash B_{j_{i}}}|\a|=0.
		\end{equation}
		It now follow (\ref{E4.9}), (\ref{E4.10}) that $\lim_{i\rightarrow\infty}\int_{B_{j_{i}}(x_{0})}\a\wedge\w=0$.
	\end{proof}
	We then have
	\begin{theorem}\label{T4.6}( \cite{Hua2} Theorem 4.3)
		Let $Cyl(X)$ be the cylinder over a compact real Killing spinor manifold $X$, $\mathcal{A}$ be a solution of $\Om$-instanton equation. If $\rho(\mathcal{A})=0$, then $\mathcal{A}$ is a flat connection.
	\end{theorem}
	\begin{proof}
		Since $\rho(\mathcal{A})=0$ and $|Tr(F_{\mathcal{A}}^{2})|\lesssim |F_{\mathcal{A}}|^{2}$, we observe that
		\begin{equation}\label{E4.1}
		\lim_{T\rightarrow\infty}\frac{1}{T}\int_{(-T,T)\times X}|Tr(F_{\mathcal{A}}^{2})|=0.
		\end{equation} 
		Since $Tr(F_{\mathcal{A}}^{2})$ is a closed $4$-form on $Cyl(X)$, it also satisfies Equation (\ref{E4.1}) and $\ast\tilde{P}$ is a $D$(bounded) $(n-4)$-form, then following Lemma \ref{L4.5}, there exist a sequence $\{j_{i}\}_{i\geq1}$ such that
		\begin{equation}\label{E4.11}
		\lim_{i\rightarrow\infty}\int_{(-j_{i},j_{i})\times X}tr(F_{\mathcal{A}}^{2})\wedge\ast\tilde{P}=0.
		\end{equation}
		It now follows (\ref{E4.11}), (\ref{E4.12}) that 
		$$\lim_{i\rightarrow\infty}\int_{(-j_{i},j_{i})\times X}|\frac{\pa A}{\pa t}|^{2}dt\wedge dvol=0,$$
		i.e., $\frac{\pa A}{\pa t}=0$. The connection $\mathcal{A}$ is not dependence on parameter $t$. Thus $$\rho(\mathcal{A})=\int_{X}|F_{A}|^{2}dvol.$$
		By the hypothesis of energy density $\rho(\mathcal{A})$, we obtain that $F_{A}=0$. We complete this proof.
	\end{proof}
	\subsection{Hodge theory on bundle $E$}
	In this section, we consider the Hodge theory on principal bundle over the complete $G_{2}$-manifold equipped with a $d$(linear) $G_{2}$-structure. At first, we recall some definitions on differential geometry. Let $E$ be a principal $G$-bundle over a complete Riemannian manifold $X$. Assume now that $d_{A}$ is a smooth connection on $E$. The formal adjoint operator of $d_{A}$ acting on $\La^{p}(X,E):=\La^{p}(X)\otimes E$ is $d^{\ast}_{A}=(\pm)\ast d_{A}\ast$.
	\begin{definition}\label{D4.5}
		The Laplace-Beltrami operator associated to $d_{A}$ is the second order operator $\De_{A}:=d_{A}d_{A}^{\ast}+d_{A}^{\ast}d_{A}$. The space of $L^{2}$-harmonic forms of degree of $p$ respect to the Laplace-Beltrami operator $\De_{A}$ is defined by
		$$H^{p}_{(2)}(X,E)=\{\a\in\La^{p}_{(2)}(X,E):\De_{A}\a=0\}.$$
	\end{definition}
	\begin{proposition}\label{P4.6}
		Let $(X,\w)$ be a complete Riemannian manifold equipped with a nonzero parallel $k$-form $\w$, $E$ be a principal $G$-bundle over $X$ and $A$ be a smooth connection on $E$. If $\w$ is $d$(linear), then $$H^{0}_{(2)}(X,E)=\{0\}.$$
		Furthermore, if the Ricci curvature is flat, $H^{1}_{(2)}(X,E)=\{0\}$.
	\end{proposition}
	\begin{proof}
		For any $\a\in H^{0}_{(2)}(X,E)$, the Weitzenb\"{o}ck formula gives:
		$$0=\langle d_{A}^{\ast}d_{A}\a,\a\rangle_{L^{2}(X)}=\langle\na_{A}^{\ast}\na_{A}\a,\a\rangle_{L^{2}(X)}=\|\na_{A}\a\|^{2}_{L^{2}(X)}.$$
		Using the Kato inequality,\ $|\na|\a||\leq|\na_{A}\a|$, we have
		$|\na|\a||=0$, i.e., $|\a|$ is a  harmonic function over $X$. Then following Corollary \ref{C1}, $|\a|\equiv0$, i.e. $\a\equiv0$.
		
		Next, we will show that if Ricci curvature is flat, $H^{1}_{(2)}(X,E)=\{0\}$. For any $\a\in H^{1}_{(2)}(X,E)$, the Weitzenb\"{o}ck formula gives:
		$$
		0=\langle \De_{A}\a,\a\rangle_{L^{2}(X)}=\langle\na_{A}^{\ast}\na_{A}\a,\a \rangle_{L^{2}(X)}=\|\na_{A}\a\|^{2}_{L^{2}(X)},
		$$
		here we use the fact the connection $A$ is flat. By Kato inequality, $\na|\a|=0$,i.e., $|\a|$ is also a harmonic function over $X$. Thus $\a\equiv 0$.
	\end{proof}
	The operator $\De_{A}$ always does not commute with $L_{\w}$, where $\w$ is parallel form on a complete manifold $X$. We cannot extend the idea of Theorem \ref{T1} to the principal bundle $E$. But on a complete $G_{2}$-manifold $X$, there exists a structure operator $C$ on $X$ (See Definition \ref{D2.1}). Then C induces isomorphisms $\La^{1}(X,E)\rightarrow\La^{2}_{7}(X,E)$. To be more specific, we can compose $\a=\a^{7}+\a^{14}$ for any $\a\in\La^{2}(X,E)$, $\a^{i}\in\La^{2}_{i}\otimes E$. There exists a one-form $\b$ such that
	\begin{equation}\label{3.7}
	C(\b):=\ast(\ast\phi\wedge\b)=\a^{7},\ i.e.,\ \b=\frac{1}{3}\big{(}\ast(\a^{7}\wedge\ast\phi)\big{)}.
	\end{equation}
	\begin{lemma}\label{3.9}
		Let $A$ be a connection on a complete $G_{2}$-manifold, $\a$ be a harmonic  $2$-form with respect to $\De_{A}$. If $X$ is non-compact, suppose also that $\a\in L^{2}$, then we have following identities:
		\begin{equation}\label{3.10}
		d^{\ast}_{A}\b=0,\ \Pi^{2}_{7}(d_{A}\b)=0.
		\end{equation}
		where $\b$ is defined as (\ref{3.7}) and $\Pi^{2}_{7}$ denote a projection map $\La^{2}\rightarrow\La^{2}_{7}$. Further more, if $A$ is a flat connection on $X$, then $\b$ is also closed with respect to $d_{A}$.
	\end{lemma}
	\begin{proof} 
		Our proof uses the author's argument in \cite{Hua3} for Yang-Mills connections. We compose $\a=\a^{7}+\a^{14}$; thus, $\a^{7}\wedge\ast\phi=\a\wedge\ast\phi$. From the identity $d_{A}\a=0$ and the fact $d\ast\phi=0$, we have
		\begin{equation}\nonumber
		0=d_{A}(\a^{7}\wedge\ast\phi)=d_{A}(\a\wedge\ast\phi)=3d_{A}\ast\b.
		\end{equation}
		Further more, using  the fact $d_{A}^{\ast}\a=d_{A}\a=0$ and $\a^{7}=\frac{1}{3}(\a+\ast(\a\wedge\phi))$, we have
		\begin{equation}\label{E4.8}
		d_{A}^{\ast}\a^{7}=\frac{1}{3}\ast d_{A}(\a\wedge\phi)=0.
		\end{equation}
		Applying operator $d^{\ast}_{A}$ to $C(\b)=\a^{7}$, following Equation (\ref{E4.8}), we get
		\begin{equation}\label{3.3}
		\ast(d_{A}\b\wedge\ast\phi)=0,\ i.e.,\ \Pi^{2}_{7}(d_{A}\b)=0.
		\end{equation}
		If $\a$ is in $L^{2}(X)$, by the definition of $\b$, we obtain that $|\b|\lesssim|\a^{7}|$, i.e., $\b$ is also in $L^{2}$. Furthermore, if $A$ is a flat connection, we have
		$$0=d_{A}^{\ast}\Pi^{2}_{7}(d_{A}\b)=d_{A}^{\ast}d_{A}\b+\ast d_{A}(d_{A}\b\wedge\phi)=d_{A}^{\ast}d_{A}\b.$$
		Then $d_{A}\b=0$. We complete this proof.
	\end{proof}
	\begin{theorem}\label{T2}
		Let $(X,\phi)$ be a complete $G_{2}$-manifold with a $d$(linear) $G_{2}$-structure $\phi$, $E$ be a principal $G$-bundle over $X$ and $A$ be a smooth connection on $E$. If $A$ is a flat connection, then $H^{p}_{(2)}(X,E)=0$ unless $p\neq 3,4$.
	\end{theorem}
	\begin{proof}
		Following Proposition \ref{P4.6}, we obtain that $H^{k}_{(2)}(X,E)=0$, $k=0,1$. We only need to  show $H^{2}_{2}(X,E)=\{0\}$. We denote $\a\in H^{2}_{(2)}(X,E)$, $\b$ is defined as (\ref{3.7}). If $A$ is a flat connection, following Lemma \ref{3.9}, $\b$ is also harmonic with respect to $\De_{A}$. By Proposition \ref{P4.6}, $\b=0$, i.e., $\a^{7}=0$. It implies that the $L^{2}$-harmonic $2$-form $\a$ also on $\La^{2}_{21}(X)\otimes E$, i.e., $\a+\ast(\a\wedge\phi)=0$. Thus we have an identity, $\|\a\|^{2}_{L^{2}(X)}=-\int_{X}tr(\a\wedge\a)\wedge\phi$. It is easy to see $tr(\a\wedge\a)$ is an closed $L^{1}$ form, following Lemma \ref{L1}, we have $\|\a\|_{L^{2}(X)}=0$, i.e., $\a=0$.
	\end{proof}
	Let us recall that from Bishop-Gromov's volume comparison theorem, we can define the asymptotic volume ratio 
	$$V_{X}:=\lim_{r\rightarrow\infty}\frac{V(r)}{r^{n}}$$
	where $V(r)$ is the volume of geodesic ball $B(r)$ centered at $p$ with radius $r$. And the above definition is independent of $p$, so we omit $p$ here. If $V_{X}>0$, we say that $(X,\rm{g})$ has maximal volume growth. We suppose that the complete manifold $X$ is Ricci-flat, then  $X$ has  maximal volume growth is equivalence to any $u\in C^{\infty}_{c}(X)$ satisfies the Sobolev inequality \cite{Sal}:
	$$\|u\|_{L^{\frac{2n}{n-2}}(X)}\lesssim\|\na u\|_{L^{2}(X)}.$$
	We then prove an useful 
	\begin{lemma}\label{L7}
		Let $(X^{n},\w)$ be a complete Ricci-flat Riemannian manifold with  maximal volume growth, $E$ be a principal $G$-bundle over $X$ and $A$ be a smooth connection on $E$. Then there is a positive constant $\de$ with following significance. If the curvature $F_{A}$ obeys
		\begin{equation}\label{E11}
		\|F_{A}\|_{L^{\frac{n}{2}}(X)}\leq\de
		\end{equation}
		then any $\a\in\La^{1}_{(2)}(X,E)$ satisfies the inequality
		$$ \|\a\|^{2}_{L^{\frac{2n}{n-2}}(X)}\leq c\langle\De_{A}\a,\a\rangle_{L^{2}(X)}.$$
		In particular,  $H^{1}_{(2)}(X,E)=\{0\}$.
	\end{lemma}
	\begin{proof}
		We observe that
		\begin{equation}\nonumber
		|\langle F_{A},[\a\wedge\a]\rangle_{L^{2}(X)}|\lesssim\|F_{A}\|_{L^{\frac{n}{2}}(X)}\|\a\|^{2}_{L^{\frac{2n}{n-2}}(X)},
		\end{equation}
		thus we have
		\begin{equation}\nonumber
		\begin{split}
		\langle \De_{A}\a,\a\rangle_{L^{2}(X)}&\geq\|\na_{A}\a\|^{2}_{L^{2}(X)}-C_{1}\|F_{A}\|_{L^{\frac{n}{2}}(X)}\|\a\|^{2}_{L^{\frac{2n}{n-2}}(X)}\\
		&\geq\|\na|\a|\|^{2}_{L^{2}(X)}-C_{1}\|F_{A}\|_{L^{\frac{n}{2}}(X)})\|\a\|^{2}_{L^{\frac{2n}{n-2}}(X)}\\
		&\geq(C_{2}-C_{1}\|F_{A}\|_{L^{\frac{n}{2}}(X)})\|\a\|^{2}_{L^{\frac{2n}{n-2}}(X)}\\
		\end{split}
		\end{equation}
		where $C_{1},C_{2}$ are positive constant only dependent on $X$. We can choose $\de$ sufficiently small to ensure that $\|F_{A}\|_{L^{\frac{n}{2}}(X)}\leq\frac{C_{2}}{2C_{1}}$, thus we complete the proof of this lemma.
	\end{proof}
	\begin{theorem}\label{T8}
		Let $(X,\phi)$ be a complete  $G_{2}$-manifold with  maximal volume growth, $E$ be a principal $G$-bundle over $X$ and $A$ be a smooth connection on $E$. If the $G_{2}$-structure $\phi$ is $d$(linear), then there is a positive constant $\de$ with following significance. If the curvature $F_{A}$ obeys 
		$$\|F_{A}\|_{L^{\frac{7}{2}}(X)}\leq\de,$$
		then
		$$H^{2}_{(2)}(X,E)=0.$$
	\end{theorem}
	\begin{proof}
		We denote $\a\in H^{2}_{(2)}(X,E)$ and $\b$ is defined as (\ref{3.7}). Then following Lemma \ref{3.3}, $\b$ satisfies
		$$0=d_{A}^{\ast}d_{A}\b+\ast([F_{A}\wedge\b]\wedge\phi).$$
		Taking the inner product of this equation with $\b$ yields
		$$0=\langle \De_{A}\b,\b\rangle_{L^{2}(X)}+\int_{X}tr(F_{A}\wedge[\b\wedge\b])\wedge\phi.$$
		For a smooth connection  $A$ with $\|F_{A}\|_{L^{\frac{7}{2}}(X)}\leq\de$, where $\de$ is a constant in the hypotheses of Lemma \ref{L7}, we have
		$$\|\b\|^{2}_{L^{\frac{14}{5}}(X)}\lesssim\langle\De_{A}\b,\b\rangle_{L^{2}(X)}.$$
		We also observe that
		\begin{equation}\nonumber
		|\int_{X}tr(F_{A}\wedge[\b\wedge\b])\wedge\phi|\lesssim\|F_{A}\|_{L^{\frac{7}{2}}(X)}\|\b\|^{2}_{L^{\frac{14}{5}}(X)}.
		\end{equation}
		Combining the preceding  inequalities gives
		\begin{equation}\nonumber
		\begin{split}
		0&\geq\langle\De_{A}\b,\b\rangle_{L^{2}(X)}-C_{3}\|F_{A}\|_{L^{\frac{7}{2}}(X)}\|\b\|^{2}_{L^{\frac{14}{5}}(X)}\\
		&\geq(C_{4}-C_{3}\|F_{A}\|_{L^{\frac{7}{2}}(X)})\|\b\|^{2}_{L^{\frac{14}{5}}(X)}.\\
		\end{split}
		\end{equation}
		where $C_{3},C_{4}$ are positive constants dependent on $X$. We can choose $\de$ sufficiently small to ensure that $\|F_{A}\|_{L^{\frac{7}{2}}(X)}\leq\frac{C_{4}}{2C_{3}}$; hence, $\b\equiv0$. It implies that $\a\in\La^{2}_{21}(X)\otimes E$. Hence following Lemma \ref{L1}, $\|\a\|^{2}_{L^{2}(X)}=-\int_{X}tr(\a\wedge\a)\wedge\phi=0$, i.e., $\a=0$.
	\end{proof}
	A connection is called a Yang-Mills connection if it is a critical point of the Yang-Mills functional $\mathrm{YM}(A)$, i.e., $d_{A}^{\ast}F_{A}=0$. In addition, all connections satisfy the Bianchi identity $d_{A}F_{A}=0$. It implies that the Yang-Mills connection is a harmonic $2$-form with respect to $\De_{A}$. There are very few gap results of Yang-Mills connection over non-compact, complete manifold, for example \cite{DM,EFM,Ger,Min}. Their results all depend on some positive conditions of Riemannian curvature tensors. Following Theorem \ref{T8}, we have a gap result for Yang-Mills connection on a complete $G_{2}$-manifold.
	\begin{corollary}
		Let $(X,\phi)$ be a complete $G_{2}$-manifold with a $d$(linear) $G_{2}$-structure $\phi$, $E$ be a principal $G$-bundle over $X$ and $A$ be a smooth Yang-Mills connection on $E$. If $X$ has maximal volume growth, then there exists a positive constant $\de\in(0,1]$ with following significance. If the curvature $F_{A}\in L^{2}(X)$ obeys
		$$\|F_{A}\|_{L^{\frac{7}{2}}(X)}\leq\de,$$
		then $A$ is a flat connection.
	\end{corollary}
	\section*{Acknowledgements}
	I would like to thank the anonymous referee for  careful reading of my manuscript and helpful comments.  I would like to thank Professor Verbitsky for kind comments regarding his article \cite{Ver}. Also I would like to thank Yuguo Qin for further discussions about this work. This work is supported by Nature Science Foundation of China No. 11801539 and Postdoctoral Science Foundation of China No. 2017M621998, No. 2018T110616.

\end{document}